\documentclass[10pt,twoside,english]{amsart}
\normalsize
\advance\oddsidemargin by -1.0cm
\advance\evensidemargin by -1.0cm
\advance\topmargin by -1.0cm 
\usepackage{amssymb}
\usepackage{babel}
\usepackage{amsmath}
\usepackage{amscd}   
\usepackage{epsfig}
\usepackage{rotating}
\usepackage{arydshln}
\usepackage{float}
\usepackage{enumerate}
\let\mathg\mathfrak
\theoremstyle{plain}
\newtheorem{cor}{Corollary}[section]
\newtheorem{lem}[cor]{Lemma}
\newtheorem{thm}[cor]{Theorem}

\newtheoremstyle{thmstylenn}
{15pt}
{5pt}
{\it}
{}
{\bf}
{ \ref{lem:sp2rep}.}
{ }
{}
\theoremstyle{thmstylenn}

\theoremstyle{definition}

\newtheorem{NB}[cor]{Remark}
\newtheorem{dfn}[cor]{Definition}

\newcommand{\Kommentar}[1]{}

\newcommand{\bdm}{\begin{displaymath}}
\newcommand{\edm}{\end{displaymath}}
\newcommand{\be}{\begin{equation}}
\newcommand{\ee}{\end{equation}}
\newcommand{\ba}[1]{\begin{array}{#1}}
\newcommand{\ea}{\end{array}}
\newcommand{\bea}[1][]{\begin{eqnarray#1}}
\newcommand{\eea}[1][]{\end{eqnarray#1}}
\newcommand{\btab}{\begin{tabular}}
\newcommand{\etab}{\end{tabular}}

\usepackage{rotating}

\newcommand{\x}{\times}

\newcommand{\ra}{\rightarrow}
\newcommand{\lra}{\longrightarrow}

\newcommand{\rk}{\ensuremath{\mathrm{rk}\,}}

\newcommand{\Id}{\ensuremath{\mathrm{Id}}}

\newcommand{\del}{\partial}

\newcommand{\ad}{\ensuremath{\mathrm{ad}}}

\newcommand{\C}{\ensuremath{\mathbb{C}}}

\newcommand{\R}{\ensuremath{\mathbb{R}}}

\newcommand{\mrm}[1]{\mathrm{#1}}
 
\newcommand{\vphi}{\ensuremath{\varphi}}    

\newcommand{\End}{\ensuremath{\mathrm{End}}}

\newcommand{\kr}{\ensuremath{\mathcal{R}}}

\newcommand{\Ad}{\ensuremath{\mathrm{Ad}\,}}

\renewcommand{\ad}{\ensuremath{\mathrm{ad}\,}}

\newcommand{\GL}{\ensuremath{\mathrm{GL}}}

\newcommand{\SL}{\ensuremath{\mathrm{SL}}}

\newcommand{\su}{\ensuremath{\mathg{su}}}
\newcommand{\SU}{\ensuremath{\mathrm{SU}}}
\newcommand{\U}{\ensuremath{\mathrm{U}}}

\newcommand{\so}{\ensuremath{\mathg{so}}}

\newcommand{\Spin}{\ensuremath{\mathrm{Spin}}}
\newcommand{\spin}{\ensuremath{\mathg{spin}}}

\renewcommand{\k}{\ensuremath{\mathfrak{k}}}
\newcommand{\g}{\ensuremath{\mathfrak{g}}}
\newcommand{\h}{\ensuremath{\mathfrak{h}}}
\newcommand{\hol}{\ensuremath{\mathfrak{hol}}}
\newcommand{\m}{\ensuremath{\mathfrak{m}}}

\newcommand{\n}{\ensuremath{\mathfrak{n}}}

\begin{document}
\def\haken{\mathbin{\hbox to 6pt{%
                 \vrule height0.4pt width5pt depth0pt
                 \kern-.4pt
                 \vrule height6pt width0.4pt depth0pt\hss}}}
    \let \hook\intprod
\setcounter{equation}{0}
%
%
\thispagestyle{empty}
%
\date{\today}
\title[Tangent Lie groups]{Tangent Lie groups are Riemannian naturally reductive spaces}
%
%
%
\author{Ilka Agricola}
\author{Ana Cristina Ferreira}
%
\address{\hspace{-5mm} 
Ilka Agricola,  
Fachbereich Mathematik und Informatik, 
Philipps-Universit\"at Marburg,
Hans-Meerwein-Stra\ss{}e,
D-35032 Marburg, Germany,
{\normalfont\ttfamily agricola@mathematik.uni-marburg.de}}
\address{\hspace{-5mm} 
Ana Cristina Ferreira,
Centro de Matem\'atica,
Universidade de Minho,
Campus de Gualtar,
4710-057 Braga, Portugal, 
{\normalfont\ttfamily anaferreira@math.uminho.pt}}
%
%
\begin{abstract}
Given a compact Lie group $G$ with Lie algebra $\g$,
we consider its tangent Lie group $TG\cong G\ltimes_{\Ad} \g$. In this short note,
we prove that $TG$ admits a left-invariant naturally reductive Riemannian metric $g$
and a metric connection with skew torsion $\nabla$ such that $(TG,g,\nabla)$ is
naturally reductive. An alternative spinorial description of the same connection
on the direct product $G\x \g$ generalizes in a rather subtle way to $TS^7$, which 
is in many senses almost a tangent Lie group.
\end{abstract}
\maketitle
\pagestyle{headings}
%
%
%
\section{Introduction}
%
%
Among all homogenous Riemannian manifolds, naturally reductive spaces are a class 
of particular
interest. Traditionally, they are defined as  Riemannian manifolds $(M=G/K,g)$ with a
reductive complement $\m$ of $\k$ in $\g$ such that
\begin{equation}\label{eq.NR}
\langle [X,Y]_\m, Z\rangle + \langle Y, [X,Z]_\m\rangle\ =\ 0 \ 
\text{ for all } X,Y,Z\in\m,
\end{equation}
where $\langle-,-\rangle$ denotes the inner product on $\m$ induced from 
$g$. For any reductive homogeneous space, the submersion $G\rightarrow G/K$ induces a 
connection that is called the \emph{canonical connection}.  It is   a
metric connection $\nabla$ with torsion $T(X,Y)=-[X,Y]_\m$ which satisfies 
$\nabla T = \nabla\kr = 0$, and  condition (\ref{eq.NR}) thus states that
a naturally reductive homogeneous space is a reductive space for which the torsion
 $T(X,Y,Z):=g(T(X,Y),Z)$ is a $3$-form on $G/K$ (see
\cite[Ch.\,X]{Kobayashi&N2} as a general reference). 
Classical examples of naturally reductive homogeneous spaces
include irreducible symmetric spaces, isotropy irreducible 
homogeneous manifolds, Lie groups with a bi-invariant metric, and
Riemannian $3$-symmetric spaces. 

In the recent article \cite{Agricola&F&F14}, the  authors together with
Thomas Friedrich (Berlin) initiated a systematic investigation and, in  dimension
six,
achieved the classification of naturally reductive homogeneous spaces. This is done
by applying recent results and techniques from
the holonomy theory of metric connections with skew torsion. Lie Groups (and spheres)
appearing in the classification always play a special role
because  they may admit several naturally reductive structures (see \cite{OR12a}, 
\cite{Olmos&R13} for details on these rather subtle points). The motivation
for this paper was to describe explicitly the naturally reductive structure
on $S^3\ltimes \R^3$  discovered in \cite{Agricola&F&F14}, and to 
investigate whether it  can be generalized. This turned out not to be as straight 
forward as expected. In fact, while a lot is known about
left-invariant naturally reductive metrics on \emph{compact} Lie groups (see 
\cite{DAtri&Z79}, \cite{Chrysikos16}), much less information is available
on \emph{non-compact} Lie groups (C.\,Gordon gives a general description of
naturally reductive nilmanifolds in \cite{Gordon85}, the more recent article
\cite{Agricola&F&S15} investigates quaternionic Heisenberg groups as
naturally reductive spaces).
Instead of the traditional approach, we shall often work with the following
definition.
\begin{dfn}
A Riemannian manifold
$(M,g)$ is said to be \emph{naturally reductive} if it is a  homogeneous space $M=G/K$
endowed with a metric connection $\nabla$ with skew torsion $T$ such that its
torsion and curvature $\kr$ are $\nabla$-parallel, 
i.\,e.~$\nabla T=\nabla \kr=0$. The connection $\nabla$ will then be called the 
\emph{Ambrose-Singer connection} or, loosely, the \emph{naturally reductive structure}
of $(M,g)$.
\end{dfn}
If $M$ is connected, complete, and simply connected, a theorem of Ambrose and Singer
asserts that the space is indeed naturally reductive in
the traditional sense \cite{Ambrose&S58}.

In this note, we prove that the tangent bundle $TG\cong G\ltimes_{\Ad} \g$ of a 
compact Lie group $G$ with Lie algebra $\g$ admits a left-invariant naturally 
reductive Riemannian metric $g$
and a metric connection with skew torsion $\nabla$ such that $(TG,g,\nabla)$ is
naturally reductive. Furthermore, we will define a suitable almost Hermitian
structure on $TG$ such that $\nabla$ is its characteristic connection
\cite{Friedrich&I1,Agricola06}.
An alternative spinorial description of the same connection
on the direct product $G\x \g$ generalizes in a rather subtle way to $TS^7$, which 
is in some sense almost a tangent Lie group.
Our construction will make use of the following well-known fact.
\begin{NB}\label{NB.struct-constants}
A compact Lie group $G$ of dimension $n$ endowed with a bi-invariant 
metric admits an orthonormal basis $e_1,\ldots, e_n$ of $\g$ such that its
structure constants $C^k_{ij}$, defined by $[e_i, e_j]\  =\  \sum_{k=1}^n C_{ij}^k e_k$,
are totally skew-symmetric in all indices. All structure constants on compact
Lie groups appearing in this paper will be chosen in this way.
\end{NB}
It is well known that $G$ itself  carries a family of naturally reductive
structures whose torsion a multiple of  $T(X,Y,Z):=\langle [X,Y],Z\rangle$---which is indeed a 
$3$-form by the structure constants' property that we just described
\cite{Cartan&Sch26a, Kobayashi&N2}.
Naively, it is this antisymmetry property that allows us to define canonical
$3$-forms on $TG$ as well, which are then candidates for
torsion tensors. However, the picture is more sophisticated than this.
We emphasize that our metric is neither a product metric, a $g$-natural metric
(see Remark \ref{NB.lifting}) nor a warped product metric or
of any `general' known type. We find it rather surprising that this
rather `asymmetric' metric has any special properties at all, and we believe that
it is worth investigating its occurrence further.

\medskip

\textbf{Acknowledgments.} Ilka Agricola 
acknowledges financial support by the
DFG within the priority programme 1388 "Representation theory".
Ana Ferreira thanks Philipps-Universit\"at Marburg for its
hospitality during a research stay in April-July 2015. Both authors
thank Reinier Storm for many intricate discussions on the topic.
%
\section{Tangent Lie groups as naturally reductive spaces}\label{sec.main}
\subsection{The Lie group $\mrm{SU}(2) \ltimes \R^3$}
%
In \cite[Theorems 8.9, 8.12]{Agricola&F&F14}, we proved the following
classification result for $6$-dimensional naturally reductive spaces and, more
generally, spaces with parallel skew torsion:
\begin{thm}\label{thm.classification-dim-6}
Let $(M^6, g, T)$ be a complete, simply connected Riemannian $6$-manifold 
with a metric connection $\nabla$ with parallel
skew torsion $T$, $\rk(* \sigma_T)=6$ and $\ker T = 0$. Then one of the 
following cases occurs:
\begin{enumerate}
\item[] Case $D.1$: $(M^6,g)$ is isometric to a nearly K\"ahler $6$-manifold.
\item[] Case $D.2$: $\hol^{\nabla} = \so(3) \subset \su(3)$, the manifold 
$(M^6,g)$ is naturally reductive, 
almost Hermitian of Gray-Hervella type  $\mathcal{W}_1 \oplus \mathcal{W}_3$, 
and isometric to one of the following 
Lie groups with a suitable family of left-invariant metrics:
\begin{enumerate}
\item The nilpotent Lie group with Lie algebra  $\R^3\x \R^3$
with commutator $[(v_1, w_1), (v_2, w_2)]\ =\ (0, v_1\x v_2)$
(see \cite{Sch07}),
\item the direct or the semidirect product of $S^3$ with  $\R^3$,  
\item the product $S^3\x S^3$  (described in Section \cite[9.4]{Agricola&F&F14}),
\item the Lie group $\SL(2,\C)$ viewed as a $6$-dimensional
real manifold  (described in Section \cite[9.5]{Agricola&F&F14}).
\end{enumerate}  
\end{enumerate}  
\end{thm}  
Case $(b)$ turned out to be rather non-intuitive and 
was therefore not described 
any further. In the proof, $S^3=\mathrm{SU}(2)$ appeared 
as the isometry group of the $3$-dimensional Euclidean space. 
We start by giving an explicit description of this space, and afterwards a 
generalisation of it. 

Recall that our Ansatz was as follows: $(M^6, g, T)$ is a complete, simply 
connected Riemannian 6-manifold with parallel skew torsion $T$ such that 
$\mrm{rk}(\ast \sigma_T)=6$ and $\mrm{ker}T =0$. In 
\cite[Theorem 8.9, case D.2]{Agricola&F&F14}, it is proved that
there is a local orthonormal 
frame  $\{e_1, \cdots, e_6\}$ such that the torsion form can be written as
\be \label{eq: torsion1}
T= \alpha\, e_{135}+ \alpha^\prime\, e_{246} + \beta\, (e_{245}+e_{236}+e_{146}), 
\mbox{ hence } \sigma_T = \beta (\beta-\alpha)\, (e_{1256} + e_{1234} + e_{3456})
\ee
and the curvature is given by
\be \label{eq: curvature2}
\kr= \beta (\alpha - \beta)\, [ (e_{35}+e_{46})^2 + (e_{15}+e_{26})^2 + (e_{13}+e_{24})^2]
\ee
with $\beta \neq 0$ and $\alpha \neq \beta$ to ensure that $\ast \sigma_T $ 
defines an almost complex structure. Our group $G$ corresponds to the parameters 
$\alpha,\alpha^\prime$ and $\beta$ 
such that $\alpha \neq 2\beta$ and $4\beta(\alpha-2\beta)-{\alpha^\prime}^2=0$. 

Here is the explicit realization of $\SU(2)\ltimes \R^3$. 
Let $\{f_1, f_2, f_3 \}$ be the 
canonical basis of $\R^3$ and $\{e_{12}, e_{13}, e_{23}\}$ be the standard basis 
of the Lie algebra $\mathfrak{su}(2) = \Lambda^2 \R^3$. Consider real 
parameters $a,b\neq 0$ and the basis of $\mathfrak{su}(2)\ltimes \R^3$ given by 
$$
x_1 = a(e_{12}, 0),\  x_3 = a(e_{13}, 0), \  x_5 = a(e_{23}, 0), \ 
x_2 = (b e_{12}, f_3),\ x_4 = (b e_{13}, - f_2),\  x_6 = (b e_{23}, f_1).
$$
with the following commutator relations given by the structure of semi-direct product
$$\begin{array}{lll}
\phantom{.} [x_1, x_3] = a x_5, & [x_2, x_4] 
= 2 b x_6 - \frac{b^2}{a} x_5, & [x_1, x_4] = [x_2,x_3] = a x_6, \\
\phantom{.} [x_5, x_1] =  a x_3, & [x_6, x_2] 
=  2b x_4 - \frac{b^2}{a} x_3, & [x_6, x_1] = [x_5, x_2] =  a x_4,\\
\phantom{.} [x_3, x_5] = a x_1, & [x_4, x_6] 
= 2b x_2 - \frac{b^2}{a}x_1, & [x_3, x_6] = [x_4, x_5] = a x_2.  \\
\end{array}$$
We define an almost Hermitian structure $(G, g, J)$ (indexed to the parameters 
$a$ and $b$) such that $x_1,\ldots, x_6 $ are an orthonormal basis of the 
Riemannian metric $g$ and the almost complex structure $J$ is such that its 
fundamental form is given by $\Omega = - (x_{12}+x_{34}+x_{56})$;
here and in the sequel, we will write $x_{ijk}$ for $x_i\wedge x_j\wedge x_k$, etc. 
We now compute 
the torsion form of the characteristic connection of $(g, J)$. The Nijenhuis tensor of 
$J$ is easily seen to be 
$$N_J = \left( a - \frac{b^2}{a} \right) x_{135} + 2 b (x_{136}+x_{145}+x_{235}) 
- 2b x_{246} + \left( \frac{b^2}{a} - a \right)(x_{146}+x_{236}+x_{245}) $$
and the twisted derivative of $\Omega$ is 
$$d^J \Omega = 3 \frac{b^2}{a} x_{135} + a (x_{245} + x_{236} + x_{146}) 
- 2b (x_{235}+x_{145}+x_{136})$$ and thus the characteristic torsion is \cite{Friedrich&I1,Agricola06}
\be \label{eq: torsion3}
T= N_J+ d^J \Omega = \left( a+ 2 \frac{b^2}{a} \right) x_{135} - 2b x_{246} 
+ \frac{b^2}{a} (x_{146}+x_{236}+x_{245}).
\ee
The connections forms are easily computed to be
\bdm
\Lambda(x_i)\ =\ \left( a+\frac{b^2}{a} \right) H_i \ \text{ for }i=1,3,5\
\text{ and } \  \Lambda(x_j)\ =\ 0 \ \text{ for }j=2,4,6,
\edm
where $H_1 = x_{35}+ x_{46}$, $- H_3 = x_{15}+x_{26}$  and $H_5 = x_{13}+x_{24}$. 
These 
elements satisfy the bracket relations $[H_1,H_3]=H_5$, $[H_5, H_1]= H_3$ and 
$[H_3, H_5]= H_1$, which means that the characteristic connection has holonomy 
algebra $\mathfrak{su}(2)$, as it should.
The curvature is then a constant multiple of the projection onto the holonomy 
algebra, namely
\be \label{eq: curvature4}
\kr = \frac{b^2}{a^2} (a^2+b^2) [H_1 \otimes H_1 + H_3 \otimes H_3 
+ H_5 \otimes H_5].
\ee
 It is a simple calculation to verify that $\nabla T = \nabla R = 0$ so, 
indeed, $G= \mathrm{SU}(2)\ltimes\R^3$ is equipped with a two-parameter family 
of naturally reductive metrics.
Comparing our explicit formulas of Eq. \ref{eq: torsion3} and 
Eq. \ref{eq: curvature4} with the Ansatz of Eq. \ref{eq: torsion1} and 
Eq. \ref{eq: curvature2}, if we take $\alpha = a + 2 \frac{b^2}{a}$, 
$\beta = \frac{b^2}{a}$ and $\alpha^\prime = 2b$ then indeed 
$(\alpha^\prime)^2 - 4 \beta (\alpha - 2 \beta) = 0$ and 
$\beta (\alpha - \beta) = \frac{b^2}{a^2}(a^2+b^2)$ so our computations check out.
\begin{NB}
In fact, more can be said about the almost complex structure of the examples occurring in
Case D.2 of Theorem \ref{thm.classification-dim-6}. The $\U(3)$ structure turns out to
be an $\SU(3)$ structure, hence it may be defined by a real spinor $\varphi$ of constant length
\cite{ACFH14}. A computer-aided computation reveals that $\varphi$ satisfies the  equation
\bdm
\nabla^g_X\varphi\ =\ \eta(X)\varphi + S(X)\varphi
\edm
with $\eta=0$ and 
\bdm
S= - \frac{\alpha'}{8} J + \frac{3\beta-\alpha}{8}\mathrm{Id}+ 
\left[ 
\begin{array}{cccccc}
-\frac{\beta+\alpha}{8} & \frac{\alpha'}{8} & 0 & 0 & 0 & 0\\
\frac{\alpha'}{8} & \frac{\beta+\alpha}{8} & 0 & 0 & 0 & 0 \\
0 & 0 & -\frac{\beta+\alpha}{8} & \frac{\alpha'}{8} & 0 & 0 \\
0 & 0 & \frac{\alpha'}{8} & \frac{\beta+\alpha}{8} & 0 & 0\\
0& 0 &0 & 0 &- \frac{\beta+\alpha}{8} & \frac{\alpha'}{8}\\ 
0& 0 & 0 & 0 &\frac{\alpha'}{8} & \frac{\beta+\alpha}{8}
\end{array}\right],
\edm
where the last summand anticommutes with $J$.
Therefore, these three summands are the $\chi_1$, $\chi_{\bar{1}}$ and $\chi_3$ components, 
respectively, of our $\mathrm{SU}(3)$-structure. An $\mathrm{SU}(3)$-structure is said 
to be half-flat if the endomorphism $S$ is symmetric -- in our example this happens if 
and only if  $\alpha'=0$.
Let $D^g$ be the Dirac operator associated to the Levi-Civita connection. We can 
readily check that
$$D^g(\varphi)  = 3\frac{\alpha-3\beta}{4}\varphi - 3\frac{\alpha'}{4}\tilde\varphi,
$$
where $\tilde\varphi$ is a second spinor linearly independent of $\varphi$.
We see that if $\alpha'=0$, then $\varphi$ is an eigenspinor, and if in addition 
$\alpha -3\beta= 0$, then $\varphi\in \mathrm{ker}D^g$. 
\end{NB}
%
\subsection{The semidirect product $TG= G \ltimes_\Ad \mathfrak{g}$}
%
Our aim is now to show how the previous example can be generalized to a
construction that starts with any compact Lie group endowed with a bi-invariant metric.

Let $N$ and $H$ be two connected Lie groups, 
$\varphi: H \ra \mathrm{Aut}(N)$ a non-trivial group
homomorphism, and assume that $N$ is abelian. The
semidirect product of $H$ and $N$ with respect 
to $\varphi$, denoted $H \ltimes_\varphi N$, is the manifold $H\times N$ endowed
with the multiplication
\bdm
(h_1, n_1)(h_2, n_2)\ =\ (h_1h_2, n_1+ \varphi(h_1)n_2 )\ 
\text{ for all }h_1, h_2\in H \text{ and } n_1, n_2 \in N. 
\edm
At the Lie algebra level, we have an induced map 
$\varphi_*: \mathfrak{h} \ra \mathrm{End}(\mathfrak{n} )$ and the bracket of 
two elements $(A, u), (B,v)\in \mathfrak{h}\oplus \mathfrak{n}$ is 
\bdm
[(A,u),(B,v)]\ =\ ([A,B], \varphi_*(A)(v)-\varphi_*(B)(u)).
\edm
Since $N$ is abelian, $\n$ is just a vector space with trivial Lie bracket.
We shall shortly construct explicitly a left-invariant metric on 
$H \ltimes_\varphi N$ for some particular choice of $H$ and $N$. Instead of 
checking manually that it is not bi-invariant, let us quickly prove that, 
more generally, $H \ltimes_\varphi N$ does not admit any bi-invariant metrics 
at all. 
%
\begin{lem}\label{lem.metric-not-bi-invariant}
The semidirect product $H\ltimes_\varphi N$  does not admit
bi-invariant Riemannian metrics.
\end{lem}
\begin{proof}
Left-invariant Riemannian metrics always exist, as they can be identified
with positive definite scalar products on $\h\oplus\n$; let $\langle,\rangle$
be one such product. Since $H\ltimes_\varphi N$ is  connected,  
$\langle,\rangle$
will be bi-invariant if and only if $\ad X$ is a skew-symmetric endomorphism
for any $X\in \h\oplus\n$. Let $A\in \h,\ u\in\n$ be such that
$\varphi_*(A)u=: v\neq 0$; this is possible, since we assumed  $\varphi$ 
non-trivial. Then
\bea[*]
0 & < & \langle (0,v), (0,v)\rangle\ =\  \langle (0,\varphi_*(A)u), (0,v)\rangle
\ =\ \langle [(A,0),(0, u)] , (0,v)\rangle\\
& =&  - \langle [(0, u),(A,0)] , (0,v)\rangle
= + \langle (A,0), [(0, u),(0,v) ] \rangle =0,
\eea[*]
since $N$ is abelian. This yields a contradiction and finishes the proof.
\end{proof}
 By a classical result of Milnor \cite{Milnor76}, bi-invariant metrics exist
only on groups isomorphic to a direct product of a compact Lie group and a
vector space (viewed as an abelian group), so actually our proof implies 
that $H \ltimes_\varphi N$ is not  isomorphic to such a product.\footnote{As an aside, we observe that 
it is not difficult to construct examples of semidirect products with $\vphi$ 
non-trivial and $N$ non-abelian admitting bi-invariant metrics; 
as the proof goes, it is clear that these
will have the property $\langle \h,[\n,\n]\rangle\neq 0$, and by Milnor's result,
they are actually isomorphic (in a non-trivial way) to direct products.}
\begin{dfn}
The Lie group $G\ltimes_\Ad \g$ obtained when choosing  $H=G$ any connected
Lie group,  $N=\g$ its Lie algebra (viewed as an abelian Lie group with 
respect to addition) and $\varphi=\Ad : G\ra \GL(\g)$ the adjoint 
representation will be called the \emph{tangent Lie group of $G$} and will
be denoted by $TG$.
\end{dfn}
Of course, $TG$ is indeed the tangent bundle of $G$.
The geometry of tangent Lie groups is described in \cite{Yano&K66-I, Yano&K66-II}
and \cite{Sekizawa86}. The Lie algebra of $TG$ is spanned by the complete lifts 
and the vertical lifts to $TG$ of the left invariant vector fields on $G$,
denoted by an upper index $c$ and $v$ respectively,
\bdm
\mathrm{Lie}(TG)\ =\ \{A^c+ B^v\, :\ A,B\in \g\},
\edm
and their commutator structure coincides with that of $\mathrm{Lie}(G\ltimes_\Ad \g)$
described above (see \cite{Yano&K66-I} for a proof),
\bdm
[A^c,B^c]\ =\ [A,B]^c,\quad [A^v,B^v]\ =\ 0,\quad
[A^c,B^v]\ =\ [A,B]^v.
\edm
Hence, given a basis $d_1,\ldots,d_n$ of $\g$, the $2n$ elements
$e_i:=d_i^c$ and $f_i:=d_i^v$ ($1\leq i\leq n$) form a basis of 
$\mathrm{Lie}(TG)\cong \g \ltimes_\ad \g$.

It is known that the complete lift of a semi-Riemannian metric $g$ on $M^n$
to $TM^n$ is a semi-Riemannian metric $g^c$ of split signature $(n,n)$, and that a 
connection $\nabla$ making $(M^n,g)$ a naturally reductive space lifts to a
connection $\nabla^c$ that turns  $(TM^n,g^c)$ into a naturally reductive space
\cite[Propositions 6.3, 7.8]{Yano&K66-I}, \cite[Theorem 3.6]{Sekizawa86}. 
Similar lifting results for constructing general \emph{Riemannian} metrics 
on $TM^n$ with a naturally reductive structure  are unknown. We shall construct such a 
metric in the case where $M^n=G$ is a Lie group.
\begin{thm}\label{thm.semidirect}
Let $G$ be a compact connected Lie group equipped with a bi-invariant metric, 
$TG= G \ltimes_\mathrm{ad} \g$ its tangent group.
Then there is a two parameter family $g_{a,b} \ (a,b\in \R^*)$ of left-invariant
Riemannian metrics on $TG$ such that 
$TG$ is naturally reductive. More precisely, there is a two parameter family 
of almost Hermitian structures of Gray-Hervella class 
$\mathcal{W}_1\oplus \mathcal{W}_3$ such that their characteristic connection 
$\nabla$, its torsion $T$, and curvature $\kr$, satisfy 
$\nabla T = \nabla \kr = 0$ and $\hol(\nabla)=[\g,\g]$.  
\end{thm}
\begin{proof}
Let $n= \dim G$ and $\{d_1, \cdots, d_n\} $ be an orthonormal basis of the 
Lie algebra $\g$ with respect to the chosen bi-invariant metric, $e_i:=d_i^c$ 
and $f_i:=d_i^v$ their complete and vertical lifts, respectively, as explained above 
($1\leq i\leq n$). We may assume that the structure constants $C_{ij}^k$ are 
totally antisymmetric in all indices, see Remark \ref{NB.struct-constants}.

We define a two-parameter family of left-invariant Riemannian metrics 
$g_{a,b}$ ($a,b \neq 0$) on $TG$ by setting the following elements of 
$\g\ltimes_\ad \g$ to be orthonormal
\bdm
x_i = a e_i \qquad \mbox{and} \qquad y_i= b e_i+ f_i \ \quad (i=1, \ldots, n).
\edm
Also, we define an almost complex structure $J$ on $M$ by the two form
\bdm
\Omega \ :=\ - (x_1\wedge y_1 + \cdots + x_n \wedge y_n).
\edm
We have the following commutator relations 
\bdm
[x_i, x_j] = a \sum_{k=1}^n C_{ij}^k x_k, \quad [x_i, y_j] = a \sum_{k=1}^n C_{ij}^k y_k, 
\quad [y_i, y_j] =  b \sum_{k=1}^n C_{ij}^k (2 y_k - \frac{b}{a} x_k).
\edm
For ease of notation, we will omit the wedge product sign in general; for instance, 
we will write $x_{ij} y_k$ instead of $x_i\wedge x_j \wedge y_k$ etc. The Nijenhuis 
tensor of the almost complex structure is the skew-symmetric tensor 
\bdm 
N\ =\  \sum_{i<j<k}^n  C_{ij}^k \left[ 
\left( a - \frac{b^2}{a} \right) ( x_{ijk} - (x_iy_{jk}+y_ix_jy_k+ y_{ij} x_k))
- 2b\,(y_{ijk} - (x_{ij}y_k+y_ix_{jk}+x_iy_jx_k)) 
\right]
\edm
and the twisted derivative of $\Omega$ is
\bdm
d^J\Omega\  =\  \sum_{i<j<k}^n  C_{ij}^k \left[ 
a \, (x_iy_ {jk}+y_ix_jy_k+y_{ij}x_k) + 3 \frac{b^2}{a}\,  x_{ijk} - 2b\,  (x_{ij}y_k + x_iy_jx_k+y_ix_{jk} )\right].
\edm
 Thus the torsion of the characteristic connection $\nabla$ 
is given by
\bdm
T\ =\ N + d^J \Omega\ =\ \sum_{i<j<k}^n  C_{ij}^k \left[ 
\left( a + 2 \frac{b^2}{a} \right)  x_ {ijk} - 2b  y_{ijk} 
+ \frac{b^2}{a} (x_iy_{jk}+y_ix_jy_k+y_{ij}x_k) \right]. 
\edm
The connection forms of the characteristic connection are 
\bdm
\Lambda(x_i) = \left( a + \frac{b^2}{a} \right) \sum_{j<k}^nC^k_{ij} (x_{jk}+y_{jk}) 
\quad  \mbox{ and } \quad \Lambda(y_i)=0 \quad \text{for } \ i=1,\ldots, n.
\edm
If we define $H_i= \sum_{j<k}^n C_{ij}^k (x_{jk}+y_{jk})$, these elements are 
linearly independent, and the Jacobi identity on 
$G$ shows that $[H_i,H_j] = \sum_{k=1}^n C_{ij}^k H_k$. The holonomy
Lie algebra $\hol(\nabla)$ is spanned by \cite[Ch.X, 4.1]{Kobayashi&N2}
\bdm
\m_0\ :=\ \{ [\Lambda(X),\Lambda(Y)] - \Lambda([X,Y])\, 
:\, X, \, Y\in \g\ltimes_\ad \g\}
\edm
and its iterated commutators with  $\Lambda(\g\ltimes_\ad \g)$. 
The only contribution to $\m_0$ comes from the elements
\bdm
[\Lambda(x_i),\Lambda(x_j)] - \Lambda([x_i,x_j])\ =\ 
\frac{b^2}{a^2} (a^2+b^2) \sum_{k=1}^n C_{ij}^k H_k,
\edm
whose coefficient in front  cannot be zero, since $a$ and $b$ are 
assumed to be non vanishing.
Hence, the characteristic connection has holonomy algebra  $[\g,\g]$. 
Furthermore, the curvature tensor is a constant multiple of the projection 
on the holonomy algebra, namely
\bdm
\kr = \frac{b^2}{a^2} (a^2+b^2)\sum_{k=1}^n  H_k \otimes H_k.
\edm
This immediately yields that both the torsion form and the curvature tensor 
are parallel w.\,r.\,t.\,$\nabla$. It remains to prove that the almost complex 
structure is of type $\mathcal{W}_1\oplus \mathcal{W}_3$. We already observed that 
$N$ is a skew symmetric so we only need to show that the $\mathcal{W}_4$ component 
vanishes, i.e. $\delta^g\Omega =0$. We use the formula (see \cite{Agricola06}, 
for instance)
\be\label{delta-omega}
\delta^g \Omega = \delta^\nabla \Omega + \frac{1}{2} (e_i \haken e_j \haken T) 
\wedge (e_i \haken e_j \haken \Omega) 
\ee
for any orthonormal basis $\{e_1,\cdots, e_{2n}\}$ of $M$. Using our adapted 
basis $\{x_1,y_1,\cdots, x_n,y_n\}$ and the fact that $\nabla$ is the characteristic 
connection of $(M,g,J)$, then we get that $\delta^g \Omega =0$.
\end{proof}
\begin{NB}\label{NB.lifting}
This is not the canonical symplectic structure on $T^\ast G$ (since $d\Omega \neq 0$).
Observe that the Hermitian form can also be written
\bdm
\Omega\ =\ - a \sum_{i=1}^n e_i \wedge f_i\ =\ - a \sum_{i=1}^n d^c_i\wedge d^v_i.
\edm
Thus, up to normalization the almost complex structure `rotates' complete lifts 
of vector fields into vertical lifts. This is similar in spirit to the almost complex structure already constructed in \cite{Dombrowski62}, which acts in a similar way on horizontal (with respect to any affine connection) and vertical lifts. 

However, the metric we constructed is not $g$-natural in any reasonable sense:
Given an affine connection $\nabla$ (usually the Levi-Civita connection of some 
Riemannian metric), recall that  a  $g$-natural metric is a metric on $TM$ for 
which  the vertical and $\nabla$-horizontal distribution are orthogonal and 
the metric coincides with the original metric on horizontal lifts. But the 
vector fields $y_i$ are not purely vertical, and there is no connection having 
the vector fields $x_i$ as horizontal lifts. 
\end{NB}
\begin{NB}\label{NB.b=0}
We excluded $b=0$ from the previous discussion; however, all formulas make sense, 
so we can investigate what happens in this limiting case. 
The value $b=0$ corresponds to the direct product metric, but on the
semidirect product Lie group, hence it is nevertheless only left-invariant, in accordance 
to Lemma \ref{lem.metric-not-bi-invariant}.  For $b=0$, we see that
$\hol(\nabla)=0$ and $\kr=0$, so the connection is flat. We will come back to 
this case in the next section.
\end{NB}
\begin{NB}
Let us look more closely at the case where $\g$ has non-trivial center $\mathfrak{z}$
of dimension $p>0$. We can assume that we chose our orthonormal basis such that
the first $p$ elements $d_1,\ldots,d_p$ span  $\mathfrak{z}$. Then 
all corresponding structure constants $C^k_{\alpha j}$ vanish, where 
$\alpha=1,\ldots,p$, and therefore $H_\alpha=0$ and 
\bdm
\mathrm{span}\,(e_\alpha, f_\beta \, |\, \alpha,\beta= 1,\ldots,p) \ =\
\ker T \ :=\ \{X\, \in \mathrm{Lie}(TG)\, |\, X\haken T=0\}.
\edm
Thus $T$ has a $2p$-dimensional kernel and the splitting theorem 
\cite[Thm 3.4]{Agricola&F&F14} implies that $TG$ is locally a Riemannian product with
vanishing torsion on one factor and torsion with trivial kernel on the other factor.
The case $\mathfrak{z}\neq \{0\}$ may therefore be reduced to lower dimensional
examples.
Furthermore, if $G$ is semisimple but not simple its Lie algebra $\g$ splits as the 
sum of $n$ ($n>1$) simple Lie algebras and clearly both the tangent group $TG$ and 
the torsion form $T$ will also split into $n$ summands, reducing again the problem 
to lower dimensional spaces. All in all, we can conclude that the interesting cases
are $G$ simple: After the space $M^6=T\SU(2)$ described before,
the next  examples are $M^{16}=T\SU(3)$ and $M^{28}=TG_2$.
\end{NB}
%
%
\subsection{The direct product $G \times \g$}\label{subsection.direct}
%
We shall now prove that the metric $\beta$ constructed on $TG$ in Theorem 
\ref{thm.semidirect} is isometric to a left-invariant metric on the
direct product $G \times \g$, despite the fact that  it is not isomorphic
(as a Lie group) to $G \ltimes \g$ by Lemma \ref{lem.metric-not-bi-invariant}.

First, let us describe the relevant metric on the direct product $G\times \g$. 
We assume, as in
Theorem  \ref{thm.semidirect}, that $G$ is a connected Lie group of dimension $n$ 
with bi-invariant metric,  $\{e_1, \cdots, e_n\} $ an orthonormal basis of $\g$
with respect to this metric, and $C_{ij}^k$ the totally antisymmetric structure 
constants of $G$. 

We define a two-parameter family of left-invariant Riemannian metrics $\tilde \beta$ 
(depending again on $a,b\in \R, \ a\neq 0$) on $G\x\g$ by setting the 
following elements of $\g\oplus\g$ to be  orthonormal,
\bdm
x_i = (a e_i,0) \qquad \mbox{and} \qquad y_i= (b e_i, e_i) \quad i=1, \ldots, n.
\edm
As an inner product on $\g\oplus\g$, this coincides of course with the inner
product defined on $\g\ltimes_\ad\g$ in Theorem \ref{thm.semidirect}. The
case $b=0$ corresponds now to the bi-invariant direct product metric on $G\x\g$.
For completeness, let us state the bracket relations given by the direct product 
structure, 
\bdm
[x_i, x_j] = a \sum_{k=1}^n C_{ij}^k x_k, \quad [x_i, y_j] = b \sum_{k=1}^n C_{ij}^k x_k, 
\quad [y_i, y_j] =  \frac{b^2}{a} \sum_{k=1}^n C_{ij}^k x_k.
\edm
Vector fields on $G\ltimes \g$ and $G\x \g$
are defined by left translation from their respective Lie algebras (identified with
the tangent space at the neutral element), so they are 
not the same on the set $G\x\g$, because the group structures are different -- as 
can be seen from the differing commutator relations. 
However, the left translation operators starting from the neutral 
element $(e,0)$ on both groups coincide. To see this, start with any point 
$\mathcal{Y}:=(h,Y)$ and 
consider the left translation in $TG$ (denoted by $L^s$) and in $G\x\g$ (denoted by 
$L^d$) by any group element $\mathcal{X}:=(g,X)$.  The letter $L$ without index
denotes the usual left translation in $G$ respectively $\g$ (the actual formula on $\g$
does not matter for the purpose here). The definition of the group 
multiplication in $G\ltimes \g$ and $G\x \g$ is equivalent to
\bdm
L^s_{\mathcal{X}} (\mathcal{Y})\ =\ (L_{g}(h), L_X (\Ad_g Y )),\quad
L^d_{\mathcal{X}} (\mathcal{Y})\ =\ (L_{g}(h), L_X(Y)).
\edm
Therefore, for $\mathcal{Y}=(e,0)$, $\Ad_g \,0 = 0$ and hence 
$L^s_{\mathcal{X}}(e,0)$ and $L^d_{\mathcal{X}}(e,0)$ coincide, and one
easily checks that their differential at $(e,0)$ coincide as well.
Therefore, left translation by 
$\mathcal{X}$ maps the origin to the \emph{same} point in the set $G\x\g$ regardless 
which group structure we consider. In particular, there is a natural identification
of tangent spaces to $TG$ and $G\x \g$ at \emph{all} points, and the metrics 
$\beta$, $\tilde \beta$ coincide in each of these tangent spaces. Recall now that
a diffeomorphism $f: TG=G\ltimes \g \ra G\x \g$ is an isometry if
$\beta_{g,X}(U,V) = \tilde \beta_{f(g,X)} (df_{g,X} U, df_{g,X}V)$. We choose $f(g,X)=(g,X)$ 
the identity, hence $df=\Id$, and this becomes thus an isometry because of the 
identification of tangent spaces and metrics.
\begin{NB}
 Within the set-up described above,
define an almost complex structure on $(G\x\g,\tilde \beta)$ by the two form 
\bdm 
\Omega \ =\  - (x_1\wedge y_1 + \cdots + x_n \wedge y_n).
\edm
It is a straightforward computation to check that, up to sign, the characteristic 
torsion of this almost complex structure is the same as the one described in the 
proof of Theorem \ref{thm.semidirect}. 
The connection forms of the characteristic connection $\nabla$ are expressed as 
\bdm
\Lambda(x_i) = - \frac{b^2}{a} \sum_{j<k}^n  C_{ij}^k (x_{jk}+y_{jk}) \qquad  \mbox{ and } 
\qquad \Lambda(y_i)= b \sum_{j<k}^n C_{ij}^k (x_{jk}+y_{jk}),
\edm
One checks that the curvature tensor is given by the same expression as the in the 
proof of Theorem \ref{thm.semidirect}.
A similar argument yields that both the torsion form and the curvature tensor are 
parallel with respect to $\nabla$.

Now, we had just seen before that the metrics on $TG$ and $G\x\g$ 
corresponding to $b=0$ are isometric (compare Remark \ref{NB.b=0}), and the connection
$\nabla$ is then flat. By a Theorem of Cartan and Schouten 
(\cite{Cartan&Sch26a}, see also 
\cite{Agricola&F10} for a modern proof), $G\ltimes\g$ with a flat metric connection
had to be isometric to a Lie group with a bi-invariant metric, so we could 
have anticipated the isometry at least in this special case.

Another argument to prove that $(G\ltimes \g, \beta)$ and $(G\times\g, \tilde\beta)$ 
are isometric is simply to use the Nomizu construction, see 
\cite[Appendix A]{Agricola&F&F14} or \cite{Tricerri&V1}, since the expressions for 
the torsion form and the curvature tensor are identical.
\end{NB}
%
\section{A peculiar connection on $TS^7$}
%
One might ask which are the crucial properties that make the construction of the metric 
connection in Section \ref{sec.main} work. Certainly, the Lie group approach is straightforward
and yields a complete description in the well-known formalism of Lie groups and
Lie algebras.

A different look at $G\x \g$ is to observe that both factors carry a remarkable
\emph{flat} metric connection --- the former one of the two flat Cartan-Schouten 
connections (for the biinvariant metric), the latter  the usual Levi-Civita connection
(for the standard euclidean metric). By the theorem of Cartan and Schouten mentioned 
before \cite{Cartan&Sch26a,Agricola&F10}, the only irreducible 
manifold carrying a flat metric connection $\nabla$ with skew torsion $T$ that is 
neither a  Lie group nor a vector space is the sphere $S^7$. Leaving flat vector spaces
aside, the crucial difference between a compact Lie group and $S^7$ lies
in the behaviour of the torsion: for Lie groups, $\nabla T=0$, while
on the $7$-sphere, the torsion $T$ fails to be  parallel \cite[p.484]{Agricola&F10}.

We shall sketch how
this point of view yields a remarkable connection on $TS^7$ and where differences appear.
First we summarize the construction of the flat metric connection with skew torsion 
on $S^7$.
In dimension $7$, the complex $\Spin(7)$-representation $\Delta^\C_7$
is the complexification of a real $8$-dimensional representation 
$\kappa: \, \Spin(7)\ra \End(\Delta_7)$, 
since the real Clifford algebra $\mathcal{C}(7)$ is isomorphic to
$\mathcal{M}(8)\oplus \mathcal{M}(8)$. Thus, we may identify  $\R^8$ with the vector space
$\Delta_7$ and embed therein the sphere $S^7$ as the set of all 
spinors of length one ($\langle\cdot,\cdot\rangle$ is the euclidean 
scalar product on 
$\Delta_7=\R^8$). Fix your favourite explicit realization of the
spin representation by skew matrices, 
$\kappa_i:=\kappa(e_i)\in\so(8)\subset \End(\R^8)$, $i=1,\ldots,7$. 
Define  vector fields $V_1,\ldots,V_7$  on $S^7$ by
\bdm
V_i(x)\ =\ \kappa_i \cdot x \text{ for }x\in S^7\subset \Delta_7.
\edm
From the antisymmetry of $\kappa_1,\ldots,\kappa_7$, one easily deduces
that the vector fields $V_1(x),\ldots, V_7(x) $ define a global orthonormal frame on $S^7$
consisting of Killing vector fields. The connection $\nabla$ on $S^7$ is then 
defined by the requirement $\nabla V_i(x)=0$.  This connection  is trivially 
flat and metric, and its torsion coefficients are given by ($i\neq j$)
\bdm
T(V_i,V_j,V_k)(x)\, =\, -\langle [V_i(x),V_j(x)],V_k\rangle\, =\,
-2\langle \kappa_i \kappa_jx,\kappa_k x\rangle\, =\,
2\langle \kappa_i\kappa_j \kappa_k x,x\rangle\, =:\, \tau_{ijk}(x) .
\edm
Of course, the coefficients $\tau_{ijk}(x)$ are not constant, reflecting that $T$
is not parallel.
The Killing orthonormal frame $V_1(x),\ldots, V_7(x)$  does \emph{not} form a Lie algebra;
rather, 
\be \label{eq:s7}
[V_i(x),V_j(x)]\, = \, [\kappa_i,\kappa_j](x)\, =\,  2\,\kappa_i\kappa_j x\, =\, 
- \sum_{k}  \tau_{ijk}(x) V_k(x) \ \ 
\text{ for } i\neq j.
\ee
The antisymmetric functions $\tau_{ijk}(x)$ replace (up to sign) 
the structure constants $C^k_{ij}$ of the Lie group approach. These commutators
$[V_i(x),V_j(x)]$ are, of course, again  Killing vector fields, but not of 
constant length; however, they span a Lie algebra isomorphic 
to $\spin(7)\subset\so(8)$. 

Now choose the frame $f_i:=\del/\del z_i$ with respect to standard euclidean 
coordinates $z_1,\ldots, z_7$ on $\R^7$, which is of course orthonormal for 
the euclidean scalar product. Formally, the vector fields $V_i$ and $f_i$ obey 
commutator relations similar to the direct product situation described in 
Subsection \ref{subsection.direct}. Introduce 
global vector fields on $S^7\x\R^7\ni (x,z)$ by
\bdm
X_i(x,z)\ =\ (aV_i(x), 0),\quad Y_i(x,z)\ =\ (b V_i(x), f_i), \quad i=1,\ldots,7 \quad (a,b \in \R, a \neq 0)
\edm
and define -- just as before  --  
a Riemannian metric $g$ on $S^7\x\R^7$ by requiring 
that  these are orthonormal, and an almost complex structure via the Hermitian form
$\Omega = - \sum X_i\wedge Y_i$. 
Similarly to the direct product case, the following bracket relations  hold
(we omit the base point $(x,z)$),
\bdm
[X_i, X_j] = - a \sum_{k=1}^n \tau_{ijk} X_k, \quad [X_i, Y_j] = - 
b \sum_{k=1}^n \tau_{ijk}\, X_k, 
\quad [Y_i, Y_j] =  -\frac{b^2}{a} \sum_{k=1}^n \tau_{ijk}\, X_k.
\edm
The Nijenhuis tensor of the almost hermitian structure is the skew-symmetric tensor 
\bdm
N \ = \left[ a- \frac{b^2}{a}  \right]\sum_{i<j<k} \tau_{ijk} 
(X_{ijk} - Y_{ij}X_k) 
+ 2b \sum_{i<j<k}\tau_{ijk} (X_{ij}Y_k - Y_{ijk}) 
\edm
and the twisted derivative of $\Omega$ is
\bdm
d^J\Omega = a \sum_{i<j<k} \tau_{ijk} X_k Y_ {ij} + 
3 \frac{b^2}{a} \sum_{i<j<k} \tau_{ijk} X_{ijk} -  2b \sum_{i<j<k} \tau_{ijk} X_{ij}Y_k.
\edm
Thus the characteristic torsion of this almost hermitian structure is given by
\bdm
\tilde T = N + d^J \Omega = \left[ a + 2 \frac{b^2}{a} \right]\sum_{i<j<k} 
\tau_{ijk}  X_ {ijk} 
- 2b\,  \sum_{i<j<k} \tau_{ijk} Y_{ijk} + \frac{b^2}{a} \sum_{i<j<k}
 \tau_{ijk} X_k Y_{ij}. 
\edm
Finally, the connection $\nabla = \nabla^g + \frac{1}{2}\tilde T$ on $S^7\x \R^7$ 
is given  by
\bdm
\nabla_{Z_i} X_j  =  \frac{b^2}{a} \sum_{k=1}^7 \tau_{ijk} Z_k 
 \qquad  \mbox{ and } \qquad 
\nabla_{Z_i} Y_j = - b  \sum_{k=1}^7 \tau_{ijk} Z_k, \ \  Z_i= X_i \text{ or } Y_i, \ \ 
i=1,\ldots,7.
\edm
If we define the $2$-forms $H_i= - \frac{1}{2}\sum_{j,k}\tau_{ijk} (X_{jk}+Y_{jk})$, we 
can summarize these identities as
\bdm
\nabla_{Z_i} X_j\ =\  \frac{b^2}{a}\, (Z_i\haken H_j),\quad
\nabla_{Z_i} Y_j\ =\  - b \, (Z_i\haken H_j).
\edm
As on $S^7$ itself, the torsion $\tilde T$ is not parallel, but 
$\nabla_X \tilde T(Y,Z,V)$ is antisymmetric in all arguments, hence the
curvature operator $\kr$ is a symmetric operator $\Lambda^2(TS^7)\ra\Lambda^2(TS^7)$
(see \cite[Remark 2.3]{Agricola06} for details on this curvature argument).  
In fact, one computes
\bdm
\kr = \frac{4 b^2}{a^2} (a^2+b^2) (H_1 \otimes H_1 + \cdots H_n \otimes H_n).
\edm
Together with the property (\ref{eq:s7}), this implies that $\nabla$ has holonomy $\spin(7)$.
%
\appendix\section{The two-fold product $G\times G$}
%
We now discuss briefly the compact case. The direct product $G\times G$  also has families of naturally reductive structures -- we 
sketch the construction to emphasize how this compares to the case described in
Section 2. This is an explicit version of the constructions from \cite{DAtri&Z79};
it  generalizes the $S^3\x S^3$ example given in 
\cite[Section 9.4]{Agricola&F&F14}, see case $(c)$ in  
Theorem \ref{thm.classification-dim-6}.
\begin{thm}
Let $G$ be a connected compact Lie group with bi-invariant metric. The group $G\x G$ can 
be equipped with a five-parameter family of left invariant naturally reductive structures. 
More precisely, $G\x G$ can be endowed with a family of almost complex structures of 
type $\mathcal{W}_1\oplus\mathcal{W}_3$ such that its characteristic connection 
satisfies $\nabla T = \nabla \kr =0$ and $\hol(\nabla)$ is either $[\g,\g]$ or trivial. 
\end{thm}
\begin{proof}
We realise $G \x G$ as the homogeneous space $K/ L$, where 
$K= G\x G \x G$ and $L= \Delta G$ is  embedded into $K$ diagonally.
Let $\k = \g\oplus \g \oplus \g$ be the Lie algebra of $K$ and 
$\Delta \g = \{ (X, X, X): X \in \g \}$ be the Lie algebra of $\Delta G$. Consider the 
following  $\Delta\g$-modules
$$\m_1 = \{ (X, aX, bX): a, b \in \R, X \in \g \},\quad
\m_2 = \{ (Y, cY, dY): c, d \in \R, Y \in \g \},\quad
\m = \m_1\oplus \m_2.
$$
One checks that $\m$ is a 
reductive complement of $\Delta\g$ inside $\k$ if and only if 
$$
\Delta :=  (a-1) (d-1) - (b-1) (c-1) \neq 0.
$$
Let $B$ denote the negative Killing form on $\g$ and 
define an inner product on $\m$, for each parameter $\lambda >0$, as
$$\langle (X_1+Y_1, a X_1 + c Y_1, b X_1 + d Y_1), 
(X_2+Y_2, a X_2 + c Y_2, b X_2 + d Y_2) \rangle
= B(X_1, X_2)+ \frac{1}{\lambda^2}  B(Y_1,Y_2).$$ 
We define also an almost complex structure $J$ on $\m$ by 
$$J((X, a X, b X)+(Y, c Y, d Y) ) = -\frac{1}{\lambda} (Y, a Y, b Y) + \lambda (X, c X, d X).$$
Let $e_1,\ldots,e_n$  be an orthonormal basis of $\g$ with antisymmetric structure 
constants $C_{ij}^k$. Consider also the elements
$$x_i=(e_i, a e_i, b e_i) \quad \mbox{ and } \quad y_i = (e_i, c e_i, d e_i),\qquad 
i=1,\ldots,n. $$
The sets $\{x_i\}$ and $\{y_i\}$ are bases of $\m_1$ and $\m_2$, respectively. 
Remark that $J$ is given, in the basis $\{x_i, y_i\}$ of $\m$, by the 2-form 
$\Omega = -(x_1\wedge y_1+ \cdots + x_n\wedge y_n)$.
Finally, let $h_i=(e_i,e_i,e_i)$. The isotropy representation 
$\lambda: \h \longrightarrow \so(\m)$ is given by 
$$\lambda(h_i)=\sum_{j<k} C_{ij}^k (x_jx_k+y_jy_k):= H_i.
$$
The structure of $\so(\m)$ together with the Jacobi identity in $\g$ imply that
$[H_i, H_j] = \sum_{k=1}^n C_{ij}^k H_k$.
The commutator structure is somewhat complicated. For ease of notation we introduce 
the following coefficients
$$\begin{array}{ll}
\alpha = -\frac{2}{\Delta} ((a^2-1)(d-1) - (b^2-1)(c-1)) \qquad 
& \beta = -\frac{2}{\Delta} (b-1)(a-1)(b-a) \\ 
\gamma = -\frac{2}{\Delta} ( a(d-b^2)+ a^2(b-d)+ (b^2-b) c) &
 \delta = -\frac{2}{\Delta} ( c (a (d-1) - bd +1) + (b-1) d) \\ 
\sigma = \frac{2}{\Delta} ( (a-1)(1-bd) + (ac-1)(b-1) )
& \tau = \frac{2}{\Delta} ( ac (d-b) + cb (1-d) + ad (b-1) )\\
\xi = - \frac{2}{\Delta} (c-1)(d-1)(c-d) 
& \eta = -\frac{2}{\Delta} ((d^2-1) (a-1) - (c^2-1)(b-1))\\
\theta = -\frac{2}{\Delta} ( d^2(c-a) + c^2(b-d) + (da-cb))
\end{array}$$
Then we can write the nonvanishing brackets of elements of $\m$ as
$$
\begin{array}{ll}
[x_i, x_j] \, = \, \sum_{k=1}^n C_{ij}^k (\mu x_k + \frac{\nu}{\lambda}y_k+\gamma h_k), & 
[x_i, y_j] \, = \, \sum_{k=1}^n C_{ij}^k (\lambda\delta x_k + \sigma y_k+\lambda\tau h_k), \\
{[y_i, y_j]} \, = \, \sum_{k=1}^n C_{ij}^k (\lambda^2\xi x_k + \lambda\eta y_k+\lambda^2\theta h_k). & \\
\end{array}
$$
The Nijenhuis tensor $N$ is totally skew-symmetric and given by 
\bea[*]
N & = & \sum_{i<j<k} C_{ij}^k \bigg[[\lambda^2 \xi + 2 \sigma - \alpha] 
(x_{ijk}-(x_iy_{jk}+y_ix_jy_k+y_{ij}x_k)) + \\
  &&  + \left. \left[ \frac{\beta}{\lambda} + \lambda (2\delta - \eta) \right] (y_{ijk}-(y_ix_{jk}+x_iy_jx_k+x_{ij}y_k))\right].
\eea[*]
We can also compute that
\bea[*]
d^c\Omega &=& \sum_{i<j<k} C_{ij}^k \left[ -3 \lambda^2 \xi x_{ijk} - 3 \frac{\beta}{\lambda} y_{ijk} + (2\sigma - \alpha) (x_iy_{jk}+y_ix_jy_k+y_{ij}x_k)\right. \\
&& +  \lambda (2 \delta - \eta) (y_ix_{jk}+x_iy_jx_k+x_{ij}y_k)\bigg].
\eea[*]
Therefore the torsion tensor $T = N+ d^c\Omega$ is given by
$$\begin{array}{lcl}
T & = & \sum_{i<j<k} C_{ij}^k \left[ [-2 \lambda^2 \xi + 2 \sigma - \alpha] x_{ijk} 
+ \left[-2 \frac{\beta}{\lambda}+ \lambda (2 \delta - \eta)\right] y_{ijk}  
- \lambda^2 \xi (x_iy_{jk}+y_ix_jy_k+y_{ij}x_k) \right.\\ 
& & \left. - \frac{\beta}{\lambda} (y_ix_{jk}+x_iy_jx_k+x_{ij}y_k)\right].
\end{array}
$$
The characteristic connection is given by the map $\Lambda: \m \lra \so(\m)$
\bdm
\Lambda(x_i) = (-\lambda^2 \xi + \sigma ) H_i, \quad   
\Lambda(y_i) = \left( -\frac{\beta}{\lambda} + \lambda \delta \right) H_i,\quad
i=1,\ldots n.
\edm
It is then a straightforward computation to check that $\Lambda T = 0$. As for the 
curvature tensor we have that 
$$\kr = \Sigma ( H_1\otimes H_1+\cdots + H_n\otimes H_n), \ \  \text{with }\ 
\Sigma := \frac{\beta^2}{\lambda^2} + \lambda^4 \xi^2 - \lambda^2 \xi (2 \sigma 
- \alpha)- \beta (2 \delta - \eta).
$$
It is then also clear that $\Lambda \kr = 0$. 
Therefore we have a $5$-parameter family of naturally reductive spaces on $G\times G$. 
We have $\kr = 0$ if and only if $\Sigma=0$;  This includes some particular cases 
like $(c=1, b=1)$, $(a=1,d=1)$ and $(c=d, a=b)$. In this case, $\hol(\nabla)=0$. 
If $\Sigma\neq 0$ then $\hol(\nabla)=[\g,\g]$. For all parameters, $\delta^g\Omega =0$, 
so the almost Hermitian structure is of type $\mathcal{W}_1\oplus \mathcal{W}_3$.
\end{proof}
\bigskip
    

\begin{thebibliography}{1111111}
%

\bibitem[Ag06]{Agricola06}
I.\,Agricola, \emph{The Srn\'{\i} lectures on non-integrable geometries with
torsion}, Arch. Math. (Brno) 42, 5--84 (2006). With an appendix by
Mario Kassuba.
%
\bibitem[ACFH15]{ACFH14}
I.\,Agricola, S.\,Chiossi, T.\,Friedrich, J.\,H\"oll, \emph{Spinorial 
description of $\mathrm{SU}(3)$- and $G_2$-manifolds}, 
J. Geom. Phys. 98 (2015), 535-555.
%
%
\bibitem[AFF15]{Agricola&F&F14}
I. Agricola, A. C. Ferreira, Th. Friedrich, \emph{The classification of 
naturaturally reductive homogeous spaces in dimensions $n\leq 6$}, 
Differential Geom. Appl. 39 (2015), 59-92.
%
\bibitem[AFS15]{Agricola&F&S15}
I. Agricola, A. C. Ferreira, R. Storm,
\emph{Quaternionic Heisenberg groups as naturally reductive homogeneous spaces},
Int. J. of Geometric Methods in Modern Physics 12 (2015), article ID 1560007. 

\bibitem[AF10]{Agricola&F10}
I.\, Agricola, Th.\,Friedrich, \emph{A note on flat metric connections with 
antisymmetric torsion}, Differ. Geom. Appl.  28 (2010), 480-487.
%


\bibitem[AS58]{Ambrose&S58}
W. Ambrose, I.\,M.~Singer, \emph{On homogeneous Riemannian manifolds},
Duke Math. J. 25, 647-669 (1958).
%
\bibitem[AZ79]{DAtri&Z79} J. E. D'Atri and W. Ziller,
\emph{Naturally reductive metrics and Einstein metrics on compact Lie
groups}, Memoirs of Amer. Math. Soc. 18 (1979).  
%
\bibitem[CSch26a]{Cartan&Sch26a}
\'E.\,Cartan and J.\,A.~Schouten, 
\emph{On the Geometry of the group manifold of simple and semisimple groups},
Proc. Amsterdam 29 (1926), 803--815.
%
\bibitem[Ch16]{Chrysikos16}
I. Chrysikos, 
\emph{Invariant connections with skew-torsion and $\nabla$-Einstein 
manifolds}, J. Lie Theory 26 (2016),  11--48.
%
\bibitem[Do62]{Dombrowski62}
P. Dombrowski, \emph{On the geometry of tangent bundles}, J. reine angew. Math.,
210 (1962), 73--88.
%
\bibitem[Go85]{Gordon85}
C. Gordon, \emph{Naturally reductive homogeneous Riemannian manifolds},
Can.\,J. Math. 37 (1985), 467--487.
%
\bibitem[FI02]{Friedrich&I1}
Th. Friedrich and S. Ivanov, \emph{Parallel spinors and connections with
  skew-symmetric torsion in string theory}, Asian Journ. Math. 6
 (2002), 303-336.
%
\bibitem[GH80]{GH80} 
A. Gray, L. M. Hervella, \emph{The Sixteen Classes of Almost Hermition
  Manifolds and Their Linear Invariants}, Ann. Mat. Pura Appl., IV. Ser. 123,
35-58 (1980).
%
%
%
%
\bibitem[KN69]{Kobayashi&N2}
S. Kobayashi and K. Nomizu, \emph{Foundations of differential geometry {II}},
  Wiley Classics Library, Wiley Inc., Princeton, 1969, 1996.
%
\bibitem[K56]{Kostant56}
B. Kostant, \emph{On differential geometry and homogeneous spaces II}, Proc. N.A.S 42 (1956), 354--357.
%
%
\bibitem[Mi76]{Milnor76}
J.\,Milnor,
\emph{Curvatures of left invariant metrics on Lie groups},
Adv. Math. 21 (1976), 293--329.
%
%
%
\bibitem[OR12]{OR12a}
C. Olmos, S. Reggiani,            
\emph{The skew-torsion holonomy theorem and naturally reductive spaces},
J. Reine Angew. Math. 664, 29-53 (2012).
%
\bibitem[OR13]{Olmos&R13}
C. Olmos, S. Reggiani, \emph{A note on the uniqueness of the canonical
connection of a naturally reductive space}, Monatsh. Math. 172, 379-386 (2013).
%
\bibitem[Sch07]{Sch07}                           
N. Schoemann,
\emph{Almost Hermitian structures with parallel torsion},
J. Geom. Phys. 57 (2007), 2187-2212.
%
\bibitem[Se86]{Sekizawa86}
M.\,Sekizawa, \emph{On complete lifts of reductive homogeneous spaces 
and generalized symmetric spaces},
Czech. Math. J. 36 (1986), 516--534.
%
%
\bibitem[TV83]{Tricerri&V1}
F.~Tricerri and L.~Vanhecke, \emph{Homogeneous structures on {R}iemannian
  manifolds}, London Math. Soc. Lecture Notes Series, vol.~83, 
Cambridge Univ.  Press, Cambridge, 1983.
%
\bibitem[YK66a]{Yano&K66-I}
K.\,Yano, S. Kobayashi, 
\emph{Prolongations of tensor fields and connections to tangent bundles I
 -- General theory}, J. Math. Soc. Japan 18 (1966), 194--210.
%
\bibitem[YK66b]{Yano&K66-II}
K.\,Yano, S. Kobayashi, 
\emph{Prolongations of tensor fields and connections to tangent bundles II
 -- Infinitesimal automorphisms}, J. Math. Soc. Japan 18 (1966), 236--246.
 %
         
\end{thebibliography}
\end{document}